\documentclass{nmd/nmd-article}
\makeatletter
 \nmd@setarticlefonts{15bp}{22bp}{17bp}{11bp}
\makeatother

\newcommand{\CPbar}{\overline{\CP}\vphantom{\CP}}

\title{Uncomputable trisection invariants of 4-manifolds}

\author{Nathan M. Dunfield}
\givenname{Nathan}
\surname{Dunfield}
\address{Dept.~of Mathematics, University of Illinois Urbana-Champaign, Urbana, IL 61801, USA
}
\email{nathan@dunfield.info}
\urladdr{http://dunfield.info}

\author{Marc Kegel}
\givenname{Marc}
\surname{Kegel}
\address{Universidad de Sevilla, Dpto.\ de Álgebra,
Avda.\ Reina Mercedes s/n,
41012 Sevilla, Spain}
\email{kegelmarc87@gmail.com}
\urladdr{https://marckegel.github.io/}

\author{Shana Yunsheng Li}
\givenname{Shana}
\surname{Li}
\address{Dept.~of Mathematics, University of Illinois Urbana-Champaign, Urbana, IL 61801, USA}
\email{yl202@illinois.edu}
\urladdr{https://shana-y-li.github.io/}

\author{Qiuyu Ren}
\givenname{Qiuyu}
\surname{Ren}
\address{Department of Mathematics, Stanford University, Stanford, CA 94305, USA}
\email{qren18@stanford.edu}
\urladdr{https://web.stanford.edu/~qren18/}

\begin{document}

\begin{abstract}
We prove that two invariants of smooth $4$-manifolds defined in terms of trisections are uncomputable: the trisection genus and the Kirby--Thompson $L$-invariant. That is, there does not exist an algorithm that takes as input a triangulated closed orientable $4$-manifold and outputs one of these quantities. The result on the $L$-invariant resolves the second part of Problem 4.116 in the K3 problem list. In the same spirit, we show that the PL multisection genus of PL manifolds is uncomputable in dimensions at least four.
\end{abstract}

\maketitle

\section{Introduction}

In~\cite{GayKirby}, Gay and Kirby showed that every closed orientable connected smooth $4$\hyp manifold $X$ admits a
\emph{trisection}: a decomposition into three $4$-dimensional $1$\hyp handlebodies such that their pairwise
intersections are $3$-dimensional handlebodies, and their triple
intersection is a closed surface $\Sigma$. This beautiful structure theorem gives rise to natural
invariants of smooth $4$-manifolds. In particular, the \emph{trisection genus} $g(X)$ is
the minimal genus of $\Sigma$ among all trisections of $X$. More subtly,
Kirby and Thompson~\cite{KirbyThompson2018} introduced a smooth $4$-manifold invariant $L_X$ obtained by minimizing a curve-complex complexity
of a trisection diagram over all trisections of $X$.
In this note, we prove that both these invariants are uncomputable. For the $L$-invariant, we have: 

\begin{theorem}\label{thm:L-inv}
  There does not exist an algorithm that takes as input a PL triangulation of a closed orientable connected $4$-manifold $X$, and outputs the Kirby--Thompson L-invariant of $X$.
\end{theorem}
This resolves the second part of Problem 4.116 in the K3 problem list \cite{K3}. 
The first part of that problem remains open, which asks whether a certain distance in the cut complex
associated to a given trisection is computable; here, $L_X$ is the minimum of this distance over all trisections of $X$. 

An analogous notion of \emph{multisections} of higher-dimensional manifolds was introduced
in~\cite{aribi2023multisections}. Recently, it was proven in~\cite{CMRZ} that every closed
orientable connected smooth $n$-manifold admits a multisection, and their
construction also applies in the PL category; see~\cite[Section~5.4]{CMRZ}.
As before, this gives rise to the
\emph{PL $(n-1)$-section genus} $g^{PL}_n(X)$ of a PL $n$-manifold $X$. For $n=4$, this agrees with the trisection genus of the corresponding smooth
$4$-manifold.\footnote{Here we are using that every PL $4$-manifold admits a smooth structure which is unique up to diffeomorphism; see for example~\cite[Theorem 2]{Milnor_overview}.}
Our next result shows that these invariants
are not computable in dimensions $n\geq4$. We formulate our result in the PL category, where finite PL
triangulations provide a natural finite input model. 

\begin{theorem}\label{thm:genus}
For any $n\ge4$, there does not exist an algorithm that takes as
input a PL triangulation of a closed orientable connected PL $n$-manifold $X$ and outputs the PL $(n-1)$-section genus of $X$.
\end{theorem}

\begin{remark}\label{rem:input}
An analogous uncomputability statement can be formulated for the smooth multisection genus once an effective finite presentation of smooth manifolds is fixed. We do not state Theorem~\ref{thm:genus} in the smooth category since there seems to exist no universally accepted finite presentation model for smooth manifolds in high dimensions.
\end{remark}

\subsection{Acknowledgments}
This project began and was carried out primarily during the 2026 Trisectors Workshop at Western Washington University. We thank the organizers of the workshop for fostering a stimulating and productive environment.  We also thank the referee for their helpful suggestions.
Dunfield and Li were partially supported by the US National Science Foundation grant DMS-2303572, and Dunfield was partially supported by the Simons Foundation award MP-TSM-00002564.
Kegel is supported by a Ram\'on y Cajal grant \mbox{(RYC2023-043251-I)} and PID2024-157173\-NB-I00 funded by MCIN/AEI/10.13039/501100011033, by ESF+, and by FEDER, EU; and by a VII Plan Propio de Investigación y Transferencia (SOL2025-36103) of the University of Sevilla.
This research was conducted during the period when Ren served as a Clay Research Fellow.  

\section{Proofs}
In this section, we prove Theorems~\ref{thm:L-inv} and~\ref{thm:genus} using classical
unrecognizability results for manifolds. The starting point is a
result of Markov asserting that there exists an integer $\ell$ such
that $\#_\ell(S^2\times S^2)$ is unrecognizable~\cite{Markov1958}. Here, we say that a smooth compact $4$-manifold $X$ is \emph{unrecognizable} if there exists no algorithm that takes as input a triangulation of a smooth compact $4$-manifold $X'$ and outputs whether or not $X'$ is diffeomorphic to $X$. Currently, the best result of this form is that of Gordon:

\begin{theorem}[\cite{Gordon2022}]\label{thm:Gordon}
The manifold $\#_{12}(S^2\times S^2)$ is smoothly unrecognizable.
\end{theorem}

Markov~\cite{Markov1958} and Gordon~\cite{Gordon2022} state their undecidability results only up to homeomorphism. However, it is straightforward to see that the proofs work in the smooth and PL categories as well; see, for example,~\cite{overview}. On the other hand, the recent improvement of Tancer~\cite{Tancer2023} applies only in the topological category.

We will prove Theorems~\ref{thm:L-inv} and~\ref{thm:genus} by showing that a hypothetical algorithm for computing the $L$-invariant or the PL $(n-1)$-section genus would contradict Theorem~\ref{thm:Gordon}, a slight refinement of it, or its higher-dimensional analogs.

\begin{proof}[Proof of Theorem~\ref{thm:L-inv}]
By \cite[Theorem~12]{KirbyThompson2018}, $L_X = 0$ if and only if $X$ is diffeomorphic to a connected sum of copies of $S^4$, $\CP^2$, $\CPbar^2$, $S^1\times S^3$, and $S^2 \times S^2$.
Thus, if $L_X > 0$, then $X$ is not
diffeomorphic to $\#_{12}(S^2 \times S^2)$. On the other hand, if $L_X=0$, then $X$ belongs to the above class of connected sums, and within this class its diffeomorphism type can be determined from its homology groups and intersection form. 
Since both are algorithmically possible (for the latter, see for example~\cite[Section 6]{4_manifold_algo}), it follows that an algorithm to compute the $L$-invariant would yield an algorithm recognizing $\#_{12} (S^2\times S^2)$, contradicting Theorem~\ref{thm:Gordon}.
\end{proof}

\begin{proof}[Proof of Theorem~\ref{thm:genus}]
It is straightforward to see that the only PL $n$-manifold with $(n-1)$-section genus $0$ is the PL $n$-sphere $S^n$. If $n\geq5$, the PL $n$-sphere is unrecognizable by Novikov; see \cite{Volodin1974,Chernavsky2006}. Thus, it is undecidable whether the PL $(n-1)$-section genus is equal to $0$ for $n\ge5$.

For $n=4$, we use Gordon's construction in \cite{Gordon2022}. From an Adian--Rabin family of finite group presentations, Gordon constructs a family of closed smooth $4$-manifolds $\{W_p\}$ with the following properties:
\begin{enumerate}
    \item there is no algorithm deciding which $W_p$'s are diffeomorphic to
$\#_{12}(S^2\times S^2)$;
    \item every $W_p$ has Euler characteristic $26$; and
    \item if $W_p$ is not diffeomorphic to $\#_{12}(S^2\times S^2)$, then $\pi_1(W_p)$ is nontrivial.
\end{enumerate}
Now $g(\#_{12}(S^2\times S^2))=24$ as $g$ is subadditive under connected sums and is bounded below by $b_2$. On the other hand, if
$\pi_1(W_p)$ is nontrivial, then the lower bound of Chu and
Tillmann~\cite{ChuTillmann2019} gives
$$g(W_p)\geq \chi(W_p)-2+3\operatorname{rk}(\pi_1(W_p)) \geq 26-2+3=27.$$
Consequently, an algorithm computing the trisection genus would determine which members of Gordon's family are diffeomorphic to $\#_{12}(S^2\times S^2)$, contradicting the defining undecidability property of that family.
\end{proof}

\begin{remark}
    More precisely, the proofs above show that there is no algorithm in dimension $4$ for deciding if $L=0$, and likewise if $g=24$. Similarly, there is no algorithm for deciding if $g_n^{PL}=0$, for $n\geq5$.

    More generally, it is undecidable whether $g=k$ for every $k\geq24$. 
    Indeed, let $\{W_p\}$ be Gordon's family of manifolds from the above proof. By considering the trisection genus of $W_p \, \# \, (\#_{k-24}\CP^2)$, one can show that an algorithm for deciding $g=k$ would yield an algorithm to decide which $W_p$'s are $\#_{12}(S^2\times S^2)$, a contradiction.

    Since we posted the first draft of the paper, it has been shown in \cite{kegel2026small} that $\#^9(S^2\times S^2)$ is unrecognizable using a family of manifolds analogous to Gordon's. This implies that it is undecidable whether $g=k$ for every $k\ge18$. We note that deciding whether $g=0$ amounts to the famous open problem of recognizing $S^4$ smoothly; see \cite[Problem~4.53]{K3}.

    One can similarly show that $g_n^{PL}=k$ is undecidable for all $n\geq5$ and $k\geq0$ by considering $g_n^{PL}\big(W_p'\,\#\,(\#_k(S^1\times S^{n-1}))\big)$ and noting that $g_n^{PL}(X)\ge \operatorname{rk}(\pi_1(X))$; here $W_p'$ is a family of manifolds among which $S^n$ is not recognizable, and with $\pi_1(W_p')\ne1$ if $W_p'$ is not $S^n$, cf.~\cite{Chernavsky2006}.
\end{remark}

\begin{remark}
Gordon showed in~\cite[Theorem 1.7]{GordonUnknottingSurfaces} that there exist smoothly
unknotted non-orientable surfaces in $S^4$ that are algorithmically unrecognizable.  
In~\cite{BCTT}, an $\mathcal L$-invariant for knotted surfaces is introduced, and it is shown that 
if a closed surface $F$ in $S^4$ satisfies $\mathcal L(F)=0$, then $F$ is
a distant sum of unknotted $2$-spheres and unknotted non-orientable
surfaces. We note that among such surfaces, homology can detect an unknotted non-orientable surface.
Consequently, an algorithm computing the $\mathcal L$-invariant for knotted surfaces would contradict Gordon's result.
\end{remark}

{\RaggedRight 
\small
\bibliographystyle{nmd/math} 
\bibliography{report}
}
\end{document}